\documentclass{article}
\usepackage{amsmath,amsthm,amsopn,amstext,amscd,amsfonts,amssymb}
\usepackage{natbib}

\def \tr {\mathop{\rm tr}\nolimits}
\def \re {\mathop{\rm Re}\nolimits}

\def \E {\mathop{\rm E}\nolimits}

\def \Vec {\mathop{\rm vec}\nolimits}

\def \Cov {\mathop{\rm Cov}\nolimits}
\def \etr {\mathop{\rm etr}\nolimits}
\def \diag {\mathop{\rm diag}\nolimits}
\def \build#1#2#3{\mathrel{\mathop{#1}\limits^{#2}_{#3}}}

\renewenvironment{abstract}
                 {\vspace{6pt}
                  \begin{center}
                  \begin{minipage}{5in}
                  \centerline{\textbf{Abstract}}
                  \noindent\ignorespaces
                 }
                 {\end{minipage}\end{center}}
\newtheorem{thm}{\textbf{Theorem}}[section]

\newtheorem{lem}{\textbf{Lemma}}[section]
\newtheorem{prop}{\textbf{Proposition}}[section]
\theoremstyle{definition}
\newtheorem{defn}{\textbf{Definition}}[section]

\newtheorem{rem}{\textbf{Remark}}[section]

\title{\Large \textbf{Moments of the Riesz distribution}}
\author{
  \textbf{Jos\'e A. D\'{\i}az-Garc\'{\i}a} \thanks{Corresponding author\newline
   {\bf Key words.}  Wishart distribution; Riesz distribution, random matrix, expectation, variance-covariance matrix.\newline
    2000 Mathematical Subject Classification. Primary 60E05, 62E15; secondary
    60E10}\\
  {\normalsize Department of Statistics and Computation} \\
  {\normalsize 25350 Buenavista, Saltillo, Coahuila, Mexico} \\
  {\normalsize E-mail: jadiaz@uaaan.mx} \\
}
\date{}
\begin{document}
\maketitle

\begin{abstract}
This article derives the first two moments of the two versions of the Riesz distribution in the
terms of their characteristic functions.
\end{abstract}

\section{Introduction}\label{sec1}

There is no doubt about the important role played by Wishart distribution in the context of
multivariate statistics and random matrix theory. Based in the Riesz measure, \citet{hl:01}
proposed a generalisation of the Wishart distribution, which they termed Riesz distribution.
Some of their main properties and related distributions have been studied by \citet{hl:01},
\citet{fk:94} and \citet{hlz:05}. Recently, \citet{dg:12} proposes two versions of the Riesz
distribution for real normed division algebras and some of their properties are also being
studied.

In particular, the characteristic function was obtained for both versions of the Riesz
distribution, but a topic that has been disregarded, is the study of its moments. As indicated
in the conclusions section, in particular, these moments can be used to study the asymptotic
normality of the Riesz distribution.

This article studies the first two moments for the two versions of the Riesz distribution.
Section \ref{sec2} reviews some definitions and notations on the matrix algebra and some
special functions with matrix arguments on real symmetric cones, also, are summarised the two
definitions of the Riesz distribution, their corresponding characteristic functions and the
section is finalised with obtaining   the first two matrix derivatives of the characteristic
functions. The main results are proposed in Section \ref{sec3}.

\section{Preliminary results}\label{sec2}

\subsection{Matrix algebra and special function with matrix argument}\label{sec21}

A detailed discussion of theory of matrices and special function with matrix argument can  be
found in \citet{mn:88} and \citet{m:82}, respectively. For convenience, we shall introduce some
notation, although in general, we adhere to standard notation forms.

Let $\mathfrak{A}^{m \times n}$ be the set of all $m \times n$ matrices over $\Re$. Let
$\mathbf{A} \in \mathfrak{A}^{m \times n}$, then $\mathbf{A}' \in \mathfrak{A}^{n \times m}$
denotes the transpose. It is denoted by ${\mathfrak S}_{m}$ the real vector space of all
$\mathbf{S} \in \mathfrak{A}^{m \times m}$ such that $\mathbf{S} = \mathbf{S}^{'}$. In
addition, let $\mathfrak{P}_{m}$ be the \emph{cone of positive definite matrices} $\mathbf{S}
\in \mathfrak{A}^{m \times m}$. Thus, $\mathfrak{P}_{m}$ consist of all matrices $\mathbf{S} =
\mathbf{X}^{'}\mathbf{X}$, with $\mathbf{X} \in \mathfrak{A}^{n \times m}$; then
$\mathfrak{P}_{m}$ is an open subset of ${\mathfrak S}_{m}$.

$\Gamma_{m}[a]$ denotes the multivariate \emph{Gamma function} for the space
$\mathfrak{S}_{m}$. This can be obtained as a particular case of the \emph{generalised gamma
function of weight $\kappa$} for the space $\mathfrak{S}_{m}$ with $\kappa = (k_{1}, k_{2},
\dots, k_{m})$, $k_{1}\geq k_{2}\geq \cdots \geq k_{m} \geq 0$, taking $\kappa =(0,0,\dots,0)$
and which for $\re(a) \geq (m-1)/2 - k_{m}$ is defined by \citep[see][]{gr:87},
\begin{eqnarray}\label{int1}
  \Gamma_{m}[a,\kappa] &=& \displaystyle\int_{\mathbf{A} \in \mathfrak{P}_{m}}
  \etr\{-\mathbf{A}\} |\mathbf{A}|^{a-(m-1)/2 - 1} q_{\kappa}(\mathbf{A}) (d\mathbf{A}) \\
&=& \pi^{m(m-1)/4}\displaystyle\prod_{i=1}^{m} \Gamma[a + k_{i}
    -(i-1)/2]\nonumber\\ \label{gammagen1}
&=& [a]_{\kappa} \Gamma_{m}[a],
\end{eqnarray}
where $\etr(\cdot) = \exp(\tr(\cdot))$, $|\cdot|$ denotes the determinant, and for $\mathbf{A}
\in \mathfrak{S}_{m}$
\begin{equation}\label{hwv}
    q_{\kappa}(\mathbf{A}) = |\mathbf{A}_{m}|^{k_{m}}\prod_{i = 1}^{m-1}|\mathbf{A}_{i}|^{k_{i}-k_{i+1}}
\end{equation}
with $\mathbf{A}_{p} = (a_{rs})$, $r,s = 1, 2, \dots, p$, $p = 1,2, \dots, m$ is termed the
\emph{highest weight vector}, see \citet{gr:87}. Also,
\begin{eqnarray*}
  \Gamma_{m}[a] &=& \displaystyle\int_{\mathbf{A} \in \mathfrak{P}_{m}}
  \etr\{-\mathbf{A}\} |\mathbf{A}|^{a-(m-1)/2 - 1}(d\mathbf{A}) \\ \label{cgamma}
    &=& \pi^{m(m-1)/4}\displaystyle\prod_{i=1}^{m} \Gamma[a-(i-1)/2],
\end{eqnarray*}
and $\re(a)> (m-1)/2$.

In other branches of mathematics the \textit{highest weight vector} $q_{\kappa}(\mathbf{A})$ is
also termed the \emph{generalised power} of $\mathbf{A}$ and is denoted as
$\Delta_{\kappa}(\mathbf{A})$, see \citet{fk:94} and \citet{hl:01}.

Additional properties of $q_{\kappa}(\mathbf{A})$, which are immediate consequences of the
definition of $q_{\kappa}(\mathbf{A})$ and the following property 1, are:
\begin{enumerate}
  \item Let $\mathbf{A} = \mathbf{L}^{*}\mathbf{DL}$ be the L'DL decomposition of $\mathbf{A} \in \mathfrak{P}_{m}^{\beta}$,
        where $\mathbf{L} \in \mathfrak{T}_{U}^{\beta}(m)$ with $l_{ii} = 1$, $i = 1, 2, \ldots ,m$ and
        $\mathbf{D} = \diag(\lambda_{1}, \dots, \lambda_{m})$, $\lambda_{i} \geq 0$, $i = 1, 2, \ldots
        ,m$. Then
        \begin{equation}\label{qk1}
          q_{\kappa}(\mathbf{A}) = \prod_{i=1}^{m} \lambda_{i}^{k_{i}}.
        \end{equation}
      \item
      \begin{equation}\label{qk2}
        q_{\kappa}(\mathbf{A}^{-1}) =  q_{-\kappa^{*}}^{*}(\mathbf{A}),
      \end{equation}
      where $\kappa^{*}=(k_{m}, k_{m-1}, \dots,k_{1})$, $-\kappa^{*}=(-k_{m}, -k_{m-1},
      \dots,-k_{1})$,
      \begin{equation}\label{hhwv}
         q_{\kappa}^{*}(\mathbf{A}) = |\mathbf{A}_{m}|^{k_{m}}\prod_{i = 1}^{m-1}|\mathbf{A}_{i}|^{k_{i}-k_{i+1}}
      \end{equation}
      and
      \begin{equation}\label{qqk1}
        q_{\kappa}^{*}(\mathbf{A}) = \prod_{i=1}^{m} \lambda_{i}^{k_{m-i+1}},
      \end{equation}
      see \citet[pp. 126-127 and Proposition VII.1.5]{fk:94}.

  Alternatively, let $\mathbf{A} = \mathbf{T}^{*}\mathbf{T}$ the Cholesky's decomposition of
  matrix $\mathbf{A} \in \mathfrak{P}_{m}^{\beta}$, with $\mathbf{T}=(t_{ij}) \in
  \mathfrak{T}_{U}^{\beta}(m)$, then $\lambda_{i} = t_{ii}^{2}$, $t_{ii} \geq 0$, $i = 1, 2,
  \ldots ,m$. See \citet[p. 931, first paragraph]{hl:01}, \citet[p. 390, lines -11 to
  -16]{hlz:05} and \citet[p.5, lines 1-6]{k:14}.
  \item if $\kappa = (p, \dots, p)$, then:
    \begin{equation}\label{qk3}
        q_{\kappa}(\mathbf{A}) = |\mathbf{A}|^{p},
    \end{equation}
    in particular if $p=0$, then $q_{\kappa}(\mathbf{A}) = 1$.
  \item if $\tau = (t_{1}, t_{2}, \dots, t_{m})$, $t_{1}\geq t_{2}\geq \cdots \geq t_{m} \geq
  0$, then:
    \begin{equation}\label{qk41}
        q_{\kappa+\tau}(\mathbf{A}) = q_{\kappa}(\mathbf{A})q_{\tau}(\mathbf{A}),
    \end{equation}
    in particular if $\tau = (p,p, \dots, p)$,  then:
    \begin{equation}\label{qk42}
        q_{\kappa+\tau}(\mathbf{A}) \equiv q_{\kappa+p}(\mathbf{A}) = |\mathbf{A}|^{p} q_{\kappa}(\mathbf{A}).
    \end{equation}
    \item Finally, for $\mathbf{B} \in \mathfrak{T}_{U}^{\beta}(m)$  in such a manner that
$\mathbf{C} =
    \mathbf{B}^{*}\mathbf{B} \in \mathfrak{S}_{m}^{\beta}$,
    \begin{equation}\label{qk5}
        q_{\kappa}(\mathbf{B}^{*}\mathbf{AB}) = q_{\kappa}(\mathbf{C})q_{\kappa}(\mathbf{A}),
    \end{equation}
see \citet[p. 776, eq. (2.1)]{hlz:08}.
\end{enumerate}
\begin{rem}
Let $\mathcal{P}(\mathfrak{S}_{m})$ denote the algebra of all polynomial functions on
$\mathfrak{S}_{m}$, and $\mathcal{P}_{k}(\mathfrak{S}_{m})$ the subspace of homogeneous
polynomials of degree $k$ and let $\mathcal{P}^{\kappa}(\mathfrak{S}_{m})$ be an irreducible
subspace of $\mathcal{P}(\mathfrak{S}_{m})$ such that:
$$
  \mathcal{P}_{k}(\mathfrak{S}_{m}) = \sum_{\kappa}\bigoplus
  \mathcal{P}^{\kappa}(\mathfrak{S}_{m}).
$$
Note that $q_{\kappa}$ is a homogeneous polynomial of degree $k$, moreover $q_{\kappa} \in
\mathcal{P}^{\kappa}(\mathfrak{S}_{m})$, for references see \citet{gr:87}.
\end{rem}
In (\ref{gammagen1}), $[a]_{\kappa}$ denotes the generalised Pochhammer symbol of weight
$\kappa$, defined as:
\begin{eqnarray*}
  [a]_{\kappa} &=& \prod_{i = 1}^{m}(a-(i-1)/2)_{k_{i}}\\
    &=& \frac{\pi^{m(m-1)/4} \displaystyle\prod_{i=1}^{m}
    \Gamma[a + k_{i} -(i-1)/2]}{\Gamma_{m}[a]} \\
    &=& \frac{\Gamma_{m}[a,\kappa]}{\Gamma_{m}[a]},
\end{eqnarray*}
where $\re(a) > (m-1)/2 - k_{m}$ and
$$
  (a)_{i} = a (a+1)\cdots(a+i-1),
$$
is the standard Pochhammer symbol.

An alternative definition of the generalised gamma function of weight $\kappa$ is proposed by
\citet{k:66}, which is defined as:%
\begin{eqnarray}\label{int2}
  \Gamma_{m}[a,-\kappa] &=& \displaystyle\int_{\mathbf{A} \in \mathfrak{P}_{m}}
    \etr\{-\mathbf{A}\} |\mathbf{A}|^{a-(m-1)/2 - 1} q_{\kappa}(\mathbf{A}^{-1})
    (d\mathbf{A}) \\
&=& \pi^{m(m-1)/4}\displaystyle\prod_{i=1}^{m} \Gamma[a - k_{i}
    -(m-i)/2] \nonumber\\ \label{gammagen2}
&=& \displaystyle\frac{(-1)^{k} \Gamma_{m}[a]}{[-a +(m-1)/2
    +1]_{\kappa}} ,
\end{eqnarray}
where $\re(a) > (m-1)/2 + k_{1}$.

Also recall that from \citet{mn:88};

\begin{enumerate}
  \item Let $\mathbf{A}$ be an $m \times n$ matrix and $\mathbf{B}$ a $p \times q$ matrix. The $mp \times
  nq$ matrix is defined as:
  $$
    \left [
        \begin{array}{ccc}
          a_{11}\mathbf{B}& \cdots & a_{1n}\mathbf{B} \\
          \vdots & \ddots & \vdots \\
          a_{m1}\mathbf{B} & \cdots & a_{mn}\mathbf{B}
        \end{array}
    \right ]
  $$
 and is termed the Kronecker product of $\mathbf{A}$ and $\mathbf{B}$ and written $\mathbf{A} \otimes
  \mathbf{B}$.
  \item Let $\mathbf{A}$ be an $m \times n$ matrix and $\mathbf{A}_{j}$ its $j$-th column, the
  $\Vec A$ is the $mn \times 1$ vector
  $$
    \Vec A =
    \left [
        \begin{array}{c}
          \mathbf{A}_{1} \\
          \mathbf{A}_{2} \\
          \vdots \\
          \mathbf{A}_{n}
        \end{array}
    \right ]
  $$
  \item $\tr\mathbf{AB} = \Vec' \mathbf{A}'\Vec\mathbf{B}$.
  \item $\tr\mathbf{ABCD} = \Vec'\mathbf{D}'(\mathbf{C}'\otimes\mathbf{A})\Vec\mathbf{B} =
  \Vec'\mathbf{D}(\mathbf{A}'\otimes\mathbf{C}')\Vec\mathbf{B}'$.
  \item For any matrix $\mathbf{X}=(x_{ij}) \in \Re^{m \times n}$, $d\mathbf{X}$ denotes the \emph{
  matrix of differentials}, $d\mathbf{X} = (dx_{ij})$.
  \item In particular\footnote{These results
  exist for more general conditions, see \citet[Theorem 1, pp. 149]{mn:88}}, if $F:\Re^{m \times m}
  \rightarrow \Re^{m \times m}$ of rank $m$, then:
  \begin{enumerate}
    \item $d|F(\mathbf{X})|^{p} = |F(\mathbf{X})|^{p}\tr (F(\mathbf{X}))^{-1} dF(\mathbf{X})$,

    \item $dF(\mathbf{X})^{-1} = -F(\mathbf{X})^{-1}dF(\mathbf{X})F(\mathbf{X})^{-1}$. Also
    \item $d\mathbf{AXB} = \mathbf{A}d\mathbf{XB}$, and
    \item $d\tr \mathbf{AXB}  = \tr \mathbf{A}d\mathbf{XB}$.
  \end{enumerate}
  \item If $f:\Re^{m \times n} \rightarrow \Re$, $\Vec' \mathbf{X} = (\Vec \mathbf{X})'$
  and $d\Vec' \mathbf{X} = (d\Vec \mathbf{X})'$ then:
   \begin{enumerate}
     \item if $d\Vec f(\mathbf{X}) = \Vec'\mathbf{B}d\Vec\mathbf{X}$, it´s obtained that:
     $$
       \frac{\partial\Vec f(\mathbf{X})}{\partial \Vec'\mathbf{X}} =  \Vec'\mathbf{B} \mbox{ and
       } \frac{\partial f(\mathbf{X})}{\partial \mathbf{X}} = \mathbf{B}.
     $$
     \item if $d^{\ 2}\Vec f(\mathbf{X}) = d\Vec'\mathbf{X}\ \mathbf{B}\ d\Vec\mathbf{X}$ then:
     $$
       \frac{\partial^{2}\Vec f(\mathbf{X})}{\partial \Vec\mathbf{X}\partial \Vec'\mathbf{X}} =
       \frac{1}{2}(\mathbf{B}+\mathbf{B}').
     $$
   \end{enumerate}
   \item Finally, note that for a matrix $\mathbf{A}$ of order $m \times m$, their submatrices
     $$
       \mathbf{A}_{1} = a_{11}, \
       \mathbf{A}_{2} =
       \left [
        \begin{array}{cc}
          a_{11} & a_{12} \\
          a_{21} & a_{22}
        \end{array}
       \right ], \ \cdots
       \mathbf{A}_{m} =
       \left [
        \begin{array}{ccc}
          a_{11} & \cdots & a_{1m} \\
          \vdots & \ddots & \vdots \\
          a_{m1} & \cdots & a_{mm}
        \end{array}
       \right ],
     $$
     can be written as:
     \begin{equation}\label{XI}
        \mathbf{A}_{i} = \mathbf{E}'_{i}\mathbf{A}\mathbf{E}_{i},
     \end{equation}
     where the columns of $\mathbf{E}_{i}: m \times i$ are the first $i$ columns of the identity matrix
     $\mathbf{I}_{m}$ and $\mathbf{E}_{i} = [\mathbf{I}_{i}\ \build{\mathbf{0}}{}{i \times m-i}]'$ is such that
     \begin{enumerate}
       \item $\mathbf{E}'_{i}\mathbf{E}_{i} = \mathbf{I}_{i}$
       \item $\mathbf{E}_{i}^{+} = \mathbf{E}_{i}^{'}$, hence, $\mathbf{E}_{i}^{+'} =
       \mathbf{E}_{i}$, where $\mathbf{A}^{+}$ denotes the Moore-Penrose inverse of
       $\mathbf{A}$,
       \item $\mathbf{E}_{m} = \mathbf{I}_{m}$,
       \item and
       $$
         \mathbf{E}_{i}\mathbf{E}'_{i} =
         \left[
            \begin{array}{c}
              \mathbf{I}_{i} \\
              \build{\mathbf{0}}{}{m-i \times i}
            \end{array}
         \right ]
         [\mathbf{I}_{i}\ \build{\mathbf{0}}{}{i \times m-i}]=
         \left [
            \begin{array}{cc}
              \mathbf{I}_{i} & \build{\mathbf{0}}{}{i \times m-i} \\
              \build{\mathbf{0}}{}{m-i \times i} & \build{\mathbf{0}}{}{m-i \times m-i}
            \end{array}
         \right ].
       $$
     \end{enumerate}
\end{enumerate}
Closely related to the differentiation are the commutation matrix $\mathbf{K}_{mn}$ and the
matrix $\mathbf{N}_{m} = \frac{1}{2}(\mathbf{I}_{m^{2}}+\mathbf{K}_{m})$.

The commutation matrix $\mathbf{K}_{mn}$ is such that for a $m \times n$ matrix $\mathbf{A}$,
$$
  \mathbf{K}_{mn} \Vec \mathbf{A} = \Vec \mathbf{A}'.
$$
with, $\mathbf{K}'_{mn} = \mathbf{K}_{mn}^{-1} = \mathbf{K}_{nm}$. If $m = n$, is written as
$\mathbf{K}_{m}$ instead of $\mathbf{K}_{mn}$. The main property of this matrix is: Let
$\mathbf{A}$ be an $m \times n$ matrix, $\mathbf{B}$ a $p \times q$ matrix. Then:
$$
  \mathbf{K}_{pm}(\mathbf{A} \otimes \mathbf{B}) = (\mathbf{B} \otimes
  \mathbf{A})\mathbf{K}_{qn}.
$$
By other hand, the matrix $\mathbf{N}_{m}$ is such that:
$$
  \mathbf{N}_{m} = \mathbf{N}'_{m} = \mathbf{N}_{m}^{2} = \mathbf{K}_{m}\mathbf{N}_{m}
  = \mathbf{N}_{m}\mathbf{K}_{m}.
$$
Note that for $\mathbf{A}$ an $m \times m$ symmetric matrix,
\begin{equation}\label{eqNm}
    \Vec \mathbf{A} = \mathbf{N}_{m}\Vec \mathbf{A}.
\end{equation}
Finally consider the following matrix factorisation.
\begin{prop}
If $\mathbf{A}$ is a non-negative definite $m \times m$ matrix then there is a upper triangular
$m \times m$ matrix, written as $\mathbf{A}^{1/2}$, such that $\mathbf{A} =
\left(\mathbf{A}^{1/2}\right)'\mathbf{A}^{1/2}$, $a_{ii} \geq 0$, $i = 1, \dots, m$.
\end{prop}

\subsection{Riesz distributions}\label{sec22}

This section summarise the densities and their corresponding characteristic functions for the
two versions of the Riesz distribution. From \citet{dg:12},

\begin{defn}\label{defnRd}
Let $\mathbf{\Sigma} \in \mathfrak{P}_{m}$ and $\kappa = (k_{1}, k_{2}, \dots, k_{m})$,
$k_{1}\geq k_{2}\geq \cdots \geq k_{m} \geq 0$.
\begin{enumerate}
  \item Then it is said that $\mathbf{X}$ has a Riesz distribution of type I if its density function is:
  \begin{equation}\label{dfR1}
    \frac{1}{\Gamma_{m}[a,\kappa] |\mathbf{\Sigma}|^{a}q_{\kappa}(\mathbf{\Sigma})}
    \etr\{-\mathbf{\Sigma}^{-1}\mathbf{X}\}|\mathbf{X}|^{a-(m-1)/2 - 1}
    q_{\kappa}(\mathbf{X})(d\mathbf{X})
  \end{equation}
  for $\mathbf{X} \in \mathfrak{P}_{m}$  and $a \in \Re$, $a \geq (m-1)/2 - k_{m}$;
  denoting this fact as $\mathbf{X} \sim \mathfrak{R}^{I}_{m}(a,\kappa,
  \mathbf{\Sigma})$.
  \item Then it is said that $\mathbf{X}$ has a Riesz distribution of type II if its density function is:
  \begin{equation}\label{dfR2}
     \frac{1}{\Gamma_{m}[a,-\kappa]|\mathbf{\Sigma}|^{a}q_{\kappa}(\mathbf{\Sigma}^{-1})}
     \etr\{-\mathbf{\Sigma}^{-1}\mathbf{X}\}|\mathbf{X}|^{a-(m-1)/2 - 1}
     q_{\kappa}(\mathbf{X}^{-1}) (d\mathbf{X})
  \end{equation}
  for $\mathbf{X} \in \mathfrak{P}_{m}$ and $a \in \Re$, $a > (m-1)/2 + k_{1}$;
  denoting this fact as $\mathbf{X} \sim \mathfrak{R}^{II}_{m}(a,\kappa,
  \mathbf{\Sigma})$.
\end{enumerate}
\end{defn}

From \citet{dg:12} and using (\ref{qk41}) it is  obtained that:

\begin{lem}\label{lem0}
Let $\mathbf{\Sigma} \in \mathfrak{P}_{m}$ and  $\kappa = (k_{1}, k_{2}, \dots, k_{m})$,
$k_{1}\geq k_{2}\geq \cdots \geq k_{m} \geq 0$.
\begin{enumerate}
  \item Then if $\mathbf{X} \sim \mathfrak{R}^{I}_{m}(a,\kappa, \mathbf{\Sigma})$ its characteristic function is:
  \begin{equation}\label{chfR1}
    \phi_{_{\mathbf{X}}}(\mathbf{T}) = q_{\kappa + a}\left(\left(\mathbf{I}_{m}-i\mathbf{\Sigma}^{1/2}
    \mathbf{T}\left(\mathbf{\Sigma}^{1/2}\right)'\right)^{-1}\right)
  \end{equation}
  for $a \in \Re$, $a \geq (m-1)/2 - k_{m}$.
  \item Then if $\mathbf{X} \sim \mathfrak{R}^{II}_{m}(a,\kappa, \mathbf{\Sigma})$ its characteristic function is:
  \begin{equation}\label{chfR2}
     \phi_{_{\mathbf{X}}}(\mathbf{T}) = q_{\kappa - a}\left(\mathbf{I}_{m}-i\mathbf{\Sigma}^{1/2}
     \mathbf{T}\left(\mathbf{\Sigma}^{1/2}\right)'\right)
  \end{equation}
  for $a \in \Re$, $a > (m-1)/2 + k_{1}$.
\end{enumerate}
\end{lem}

\subsection{Differentiation}\label{sec23}

Finally, consider the following result about differentiation.

\begin{lem}\label{lem1}
Let $\mathbf{\Sigma} \in \mathfrak{P}_{m}$ and  $\kappa = (k_{1}, k_{2}, \dots, k_{m})$,
$k_{1}\geq k_{2}\geq \cdots \geq k_{m} \geq 0$. Then if $\mathbf{X} \sim
\mathfrak{R}^{I}_{m}(a,\kappa, \mathbf{\Sigma})$ it is obtained:
\begin{enumerate}
  \item $\displaystyle \frac{\partial \Vec \phi_{_{\mathbf{X}}}(\mathbf{T})}{\partial \Vec' \mathbf{T}} $
  $$
    = i\left [\sum_{i = 1}^{m}(t_{i}-t_{i+1})\Vec'\mathbf{A}_{i}
    \right]\mathbf{N}_{m} q_{\tau}\left(\left (\mathbf{I}_{m}-
  i \mathbf{\Sigma}^{1/2}\mathbf{T}\left(\mathbf{\Sigma}^{1/2}\right)'\right)^{-1}\right),
  $$
  \item and $\displaystyle \frac{\partial^{2} \Vec \phi_{_{\mathbf{X}}}(\mathbf{T})}{\partial \Vec
  \mathbf{T}\partial \Vec' \mathbf{T}}$
  \begin{eqnarray*}
     &=&
    i^{2}\mathbf{N}_{m}\left \{\sum_{i = 1}^{m}(t_{i}-t_{i+1})^{2} \Vec\mathbf{A}_{i} \Vec'\mathbf{A}_{i}\right.\\
      &+& \frac{1}{2}\left [\build{\sum\sum}{m}{i \neq j}(t_{i}-t_{i+1})(t_{j}-t_{j+1})
      \Vec \mathbf{A}_{i} \Vec'\mathbf{A}_{j}\right.\\
      &+& \build{\sum\sum}{m}{i \neq j}(t_{i}-t_{i+1})(t_{j}-t_{j+1})
      \Vec \mathbf{A}_{j}\Vec'\mathbf{A}_{i}\Bigg ]\\
      &+& \sum_{i=1}^{m}(t_{i}-t_{i+1})\left(\mathbf{A}_{i}\otimes \mathbf{A}_{i}\right)\Bigg\}\mathbf{N}_{m}
      q_{\tau}\left(\left (\mathbf{I}_{m}- i \mathbf{\Sigma}^{1/2}\mathbf{T}\left(\mathbf{\Sigma}^{1/2}\right)'\right)^{-1}\right),\\
  \end{eqnarray*}
\end{enumerate}
where
$$
  \mathbf{A}_{\alpha} = \left(\mathbf{\Sigma}^{1/2}\right)'\mathbf{E}_{\alpha}\left(\mathbf{E}'_{\alpha}\left (\mathbf{I}_{m}-
  i \mathbf{\Sigma}^{1/2}\mathbf{T}\left(\mathbf{\Sigma}^{1/2}\right)'\right)\mathbf{E}_{\alpha}\right)^{-1}
\mathbf{E}'_{\alpha}\mathbf{\Sigma}^{1/2}, \quad \alpha = i, j,
$$
$\tau = \kappa+a = (t_{1}, \dots, t_{m})$ and $a \geq (m-1)/2 - k_{m}$.
\end{lem}
\begin{proof}
By 6(a), 6(b) and Lemma \ref{lem0}, defining $\kappa+a = \tau = (t_{1}, \dots, t_{m})$, with
$t_{m+1}= 0$, $\mathbf{E}_{m} = \mathbf{I}_{m}$ and
$$
  \mathbf{A}_{\alpha} = \left(\mathbf{\Sigma}^{1/2}\right)'\mathbf{E}_{\alpha}\left(\mathbf{E}'_{\alpha}\left (\mathbf{I}_{m}-
  i \mathbf{\Sigma}^{1/2}\mathbf{T}\left(\mathbf{\Sigma}^{1/2}\right)'\right)\mathbf{E}_{\alpha}\right)^{-1}
\mathbf{E}'_{\alpha}\mathbf{\Sigma}^{1/2}, \quad \alpha = i, j,
$$
it is obtained that
$$
  d\phi_{_{\mathbf{X}}}(\mathbf{T}) = i\left [\sum_{i = 1}^{m}(t_{i}-t_{i+1})\tr\mathbf{A}_{i}
    d\mathbf{T}\right] q_{\tau}\left(\left (\mathbf{I}_{m}-
  i \mathbf{\Sigma}^{1/2}\mathbf{T}\left(\mathbf{\Sigma}^{1/2}\right)'\right)^{-1}\right).
$$
Hence, vectorising,
$$
  d\Vec\phi_{_{\mathbf{X}}}(\mathbf{T}) = i\left [\sum_{i = 1}^{m}(t_{i}-t_{i+1})\Vec'\mathbf{A}_{i}
    \right] q_{\tau}\left(\left (\mathbf{I}_{m}-
  i \mathbf{\Sigma}^{1/2}\mathbf{T}\left(\mathbf{\Sigma}^{1/2}\right)'\right)^{-1}\right)d\Vec \mathbf{T},
$$
from where applying 7(a) and (\ref{eqNm}), the desired result is obtained.

Analogously, differentiating again,
\begin{eqnarray*}
    d^{2}\phi_{_{\mathbf{X}}}(\mathbf{T}) &=&
    i^{2}\left \{\sum_{i = 1}^{m}(t_{i}-t_{i+1})^{2} \tr\mathbf{A}_{i}d\mathbf{T} \tr\mathbf{A}_{i}d\mathbf{T}\right.\\
      &+& \build{\sum\sum}{m}{i \neq j}(t_{i}-t_{i+1})(t_{j}-t_{j+1})
      \tr \mathbf{A}_{i}d\mathbf{T} \tr\mathbf{A}_{j}d\mathbf{T}\\
      &+& \sum_{i=1}^{m}(t_{i}-t_{i+1})\tr\mathbf{A}_{i}d\mathbf{T} \mathbf{A}_{i}d\mathbf{T}\Bigg\}
      q_{\tau}\left(\left (\mathbf{I}_{m}- i \mathbf{\Sigma}^{1/2}\mathbf{T}\left(\mathbf{\Sigma}^{1/2}\right)'\right)^{-1}\right).\\
  \end{eqnarray*}
And vectorising using 4, it is got;
\begin{eqnarray*}
    d^{2}\Vec\phi_{_{\mathbf{X}}}(\mathbf{T}) &=&
    i^{2}\left \{\sum_{i = 1}^{m}(t_{i}-t_{i+1})^{2} d\Vec' \mathbf{T}\Vec\mathbf{A}_{i}\Vec'\mathbf{A}_{i}d\Vec \mathbf{T}\right.\\
\end{eqnarray*}
\begin{eqnarray*}
      &+& \build{\sum\sum}{m}{i \neq j}(t_{i}-t_{i+1})(t_{j}-t_{j+1})
      d\Vec' \mathbf{T} \Vec\mathbf{A}_{i}\Vec'\mathbf{A}_{j}d\Vec \mathbf{T}\\
      &+& \sum_{i=1}^{m}(t_{i}-t_{i+1})d\Vec' \mathbf{T}\left(\mathbf{A}_{i}\otimes \mathbf{A}_{i}\right)d\Vec \mathbf{T}\Bigg\}
      q_{\tau}\left(\left (\mathbf{I}_{m}- i \mathbf{\Sigma}^{1/2}\mathbf{T}\left(\mathbf{\Sigma}^{1/2}\right)'\right)^{-1}\right).\\
\end{eqnarray*}
From where, by applying 7(b) and (\ref{eqNm}), the outcome sought is obtained. \qed
\end{proof}

Analogously, one has:

\begin{lem}\label{lem2}
Let $\mathbf{\Sigma} \in \mathfrak{P}_{m}$ and  $\kappa = (k_{1}, k_{2}, \dots, k_{m})$,
$k_{1}\geq k_{2}\geq \cdots \geq k_{m} \geq 0$. Then if $\mathbf{X} \sim
\mathfrak{R}^{II}_{m}(a,\kappa, \mathbf{\Sigma})$ it is obtained:
\begin{enumerate}
  \item $\displaystyle \frac{\partial \Vec \phi_{_{\mathbf{X}}}(\mathbf{T})}{\partial \Vec' \mathbf{T}}$
  $$
    = -i\left [\sum_{i = 1}^{m}(t_{i}-t_{i+1})\Vec'\mathbf{A}_{i}
    \right]\mathbf{N}_{m} q_{\tau}\left(\mathbf{I}_{m}-
  i \mathbf{\Sigma}^{1/2}\mathbf{T}\left(\mathbf{\Sigma}^{1/2}\right)'\right),
  $$
  \item and $\displaystyle \frac{\partial^{2} \Vec \phi_{_{\mathbf{X}}}(\mathbf{T})}{\partial \Vec
  \mathbf{T}\partial \Vec' \mathbf{T}}$
  \begin{eqnarray*}
     &=&
    i^{2}\mathbf{N}_{m}\left \{\sum_{i = 1}^{m}(t_{i}-t_{i+1})^{2} \Vec\mathbf{A}_{i} \Vec'\mathbf{A}_{i}\right.\\
      &+& \frac{1}{2}\left [\build{\sum\sum}{m}{i \neq j}(t_{i}-t_{i+1})(t_{j}-t_{j+1})
      \Vec \mathbf{A}_{i} \Vec'\mathbf{A}_{j}\right.\\
      &+& \build{\sum\sum}{m}{i \neq j}(t_{i}-t_{i+1})(t_{j}-t_{j+1})
      \Vec \mathbf{A}_{j}\Vec'\mathbf{A}_{i}\Bigg ]\\
      &-& \sum_{i=1}^{m}(t_{i}-t_{i+1})\left(\mathbf{A}_{i}\otimes \mathbf{A}_{i}\right)\Bigg\}\mathbf{N}_{m}
      q_{\tau}\left(\mathbf{I}_{m}- i \mathbf{\Sigma}^{1/2}\mathbf{T}\left(\mathbf{\Sigma}^{1/2}\right)'\right),\\
  \end{eqnarray*}
\end{enumerate}
where:
$$
  \mathbf{A}_{\alpha} = \left(\mathbf{\Sigma}^{1/2}\right)'\mathbf{E}_{\alpha}\left(\mathbf{E}'_{\alpha}\left (\mathbf{I}_{m}-
  i \mathbf{\Sigma}^{1/2}\mathbf{T}\left(\mathbf{\Sigma}^{1/2}\right)'\right)\mathbf{E}_{\alpha}\right)^{-1}
\mathbf{E}'_{\alpha}\mathbf{\Sigma}^{1/2}, \quad \alpha = i, j,
$$
$\tau = \kappa-a = (t_{1}, \dots, t_{m})$ and $a > (m-1)/2 + k_{1}$.
\end{lem}
\begin{proof}This is a verbatim copy of the proof of Lemma \ref{lem1}. \qed
\end{proof}

\section{Moments of Riesz distributions}\label{sec3}

This section proposed the main result.

\begin{thm}\label{Moments}
Let $\mathbf{\Sigma} \in \mathfrak{P}_{m}$ and  $\kappa = (k_{1}, k_{2}, \dots, k_{m})$,
$k_{1}\geq k_{2}\geq \cdots \geq k_{m} \geq 0$.
\begin{enumerate}
  \item Then if $\mathbf{X}$ has a Riesz distribution of type I,
    \begin{enumerate}
      \item $\E(\mathbf{X}) = (k_{m} + a)\mathbf{\Sigma} + \displaystyle\sum_{i = 1}^{m-1}(k_{i} - k_{i-1})
      \left(\mathbf{\Sigma}^{1/2}\right)'\mathbf{E}_{i}\mathbf{E}'_{i}\mathbf{\Sigma}^{1/2}$.
      \item and $\Cov(\Vec \mathbf{X})$ is
      $$
        (k_{m} + a)\mathbf{N}_{m}(\mathbf{\Sigma} \otimes \mathbf{\Sigma})
         + \sum_{i = 1}^{m-1}(k_{i} - k_{i-1})\mathbf{N}_{m}\left(\left(\mathbf{\Sigma}^{1/2}\right)'\mathbf{E}_{i}
        \mathbf{E}'_{i}\mathbf{\Sigma}^{1/2}\otimes \left(\mathbf{\Sigma}^{1/2}\right)'\mathbf{E}_{i}\mathbf{E}'_{i}
        \mathbf{\Sigma}^{1/2}\right).
      $$
    \end{enumerate}
  for $a \geq (m-1)/2 - k_{m}$.
  \item Then if $\mathbf{X}$ has a Riesz distribution of type II,
  \begin{enumerate}
      \item $\E(\mathbf{X}) = -(k_{m} - a)\mathbf{\Sigma} - \displaystyle\sum_{i = 1}^{m-1}(k_{i} - k_{i-1})
      \left(\mathbf{\Sigma}^{1/2}\right)'\mathbf{E}_{i}\mathbf{E}'_{i}\mathbf{\Sigma}^{1/2}$.
      \item and $\Cov(\Vec \mathbf{X})$ is
      \begin{footnotesize}
      $$
        -(k_{m} - a)\mathbf{N}_{m}(\mathbf{\Sigma} \otimes \mathbf{\Sigma})
         - \sum_{i = 1}^{m-1}(k_{i} - k_{i-1})\mathbf{N}_{m}\left(\left(\mathbf{\Sigma}^{1/2}\right)'\mathbf{E}_{i}
        \mathbf{E}'_{i}\mathbf{\Sigma}^{1/2}\otimes \left(\mathbf{\Sigma}^{1/2}\right)'\mathbf{E}_{i}\mathbf{E}'_{i}
        \mathbf{\Sigma}^{1/2}\right).
      $$
      \end{footnotesize}
    \end{enumerate}
  for $a > (m-1)/2 + k_{1}$.
\end{enumerate}
\end{thm}

\begin{proof}
Results are immediately from lemmas \ref{lem1} and \ref{lem2}, remembering that:
$$
  \E(\Vec\mathbf{X}) = \left .\frac{\partial \Vec \phi_{_{\mathbf{X}}}(\mathbf{T})}{i\ \partial\Vec
  \mathbf{X}}\right |_{\mathbf{T} = 0}
$$
and
\begin{equation}\label{cov}
    \Cov(\Vec\mathbf{X}) = \E(\Vec \mathbf{X}\Vec' \mathbf{X})-\E(\Vec \mathbf{X})\E(\Vec' \mathbf{X})
\end{equation}
where
$$
  \E(\Vec \mathbf{X}\Vec' \mathbf{X}) = \left .\frac{\partial^{2} \Vec \phi_{_{\mathbf{X}}}(\mathbf{T})}{i^{2}\ \partial\Vec
  \mathbf{X} \partial\Vec' \mathbf{X}}\right |_{\mathbf{T} = 0}.
$$
In order to ensure that $\Cov(\Vec\mathbf{X}) = \Cov'(\Vec\mathbf{X})$ it is necessary that
$$
  \E(\Vec \mathbf{X})\E'(\Vec \mathbf{X})
$$
be a symmetric matrix. Then, proceeding as in the case of $ \E(\Vec \mathbf{X}\Vec' \mathbf{X})
= (\mathbf{B}+\mathbf{B}')/2$, consider the following equivalent definition;
$$
  \Cov(\Vec\mathbf{X}) = \E(\Vec \mathbf{X}\Vec' \mathbf{X})-\frac{1}{2}\{\E(\Vec \mathbf{X})\E(\Vec'
  \mathbf{X})+ [\E(\Vec \mathbf{X})\E(\Vec' \mathbf{X})]'\},
$$
which alternative definition, coincides with (\ref{cov}) when $\E(\Vec \mathbf{X})\E'(\Vec
\mathbf{X})$ is a symmetric matrix and at the same time ensuring  that $\Cov(\Vec\mathbf{X})$
is  a symmetric matrix. \qed
\end{proof}

Observe that in Theorem \ref{Moments}.1 and \ref{Moments}.2, are defined as $a = n/2$,
$\mathbf{\Sigma} \rightarrow 2\mathbf{\Sigma}$ and $\kappa = (0, \dots,0)$ the Wishart case is
obtained. Moreover,
\begin{enumerate}
  \item $\E(\mathbf{X}) = n \mathbf{\Sigma}$,
  \item $\Cov(\Vec \mathbf{X}) = 2n \mathbf{N}_{m}(\mathbf{\Sigma} \otimes \mathbf{\Sigma}) =
  n(\mathbf{I}_{m^{2}}+\mathbf{K}_{m})(\mathbf{\Sigma} \otimes \mathbf{\Sigma})$,
\end{enumerate}
see \citet[p. 90]{m:82} and \citet[p. 253]{mn:88}.

\section{Conclusions}

Observe that if $\mathbf{Y}_{1}, \cdots, \mathbf{Y}_{N}$, are independent $p$-dimensional
random vectors, with $N \geq p$, such that the random matrix $\mathbf{X}$, is defined as:
$$
  \mathbf{X} = \sum_{i = 1}^{N}\mathbf{Y}_{i}\mathbf{Y}'_{i} = \mathbf{Y}'\mathbf{Y}, \quad
  \mathbf{Y}' = (\mathbf{Y}_{1}, \cdots, \mathbf{Y}_{N})
$$
has a Riesz distribution type I, then by multivariate central limit theorem, \citet[p.
15]{m:82} and Theorem \ref{Moments}, if
$$
  \bar{\mathbf{Y}} = \frac{1}{N}\sum_{i = 1}^{N}\mathbf{Y}_{i}
$$
and
$$
  \mathbf{S}(n) = \frac{1}{n}\sum_{i = 1}^{N} (\mathbf{Y}_{i} - \bar{\mathbf{Y}})(\mathbf{Y}_{i} -
  \bar{\mathbf{Y}})', \quad N = n+1
$$
the asymptotic distribution as $n \rightarrow \infty$ of
$$
  n^{1/2}\left[\Vec \mathbf{S}(n) - \frac{(k_{m} + a)}{n}\Vec \mathbf{\Sigma} + \sum_{i = 1}^{m-1}\frac{(k_{i} - k_{i-1})}{n}
      \Vec \left(\mathbf{\Sigma}^{1/2}\right)'\mathbf{E}_{i}\mathbf{E}'_{i}\mathbf{\Sigma}^{1/2}\right]
$$
is:
\begin{footnotesize}
$$
  \mathcal{N}_{m^{2}}\left(\mathbf{0}, \frac{(k_{m} + a)}{n}\mathbf{N}_{m}(\mathbf{\Sigma} \otimes \mathbf{\Sigma})
         + \sum_{i = 1}^{m-1}\frac{(k_{i} - k_{i-1})}{n}\mathbf{N}_{m}\left(\left(\mathbf{\Sigma}^{1/2}\right)'\mathbf{E}_{i}
        \mathbf{E}'_{i}\mathbf{\Sigma}^{1/2}\otimes \left(\mathbf{\Sigma}^{1/2}\right)'\mathbf{E}_{i}\mathbf{E}'_{i}
        \mathbf{\Sigma}^{1/2}\right)\right).
$$
\end{footnotesize}
Note that this asymptotic multivariate normal distribution is singular, moreover, its rank is
$m(m+1)/2$. Also, observe that if $a = n/2$, $\mathbf{\Sigma} \rightarrow 2\mathbf{\Sigma}$ and
$\kappa = (0, \dots,0)$ the asymptotic result is obtained by \citet[pp. 90-91]{m:82} for the
Wishart case. The author is currently studying in detail the distribution of the random matrix
$\mathbf{Y}$ and some of their basic properties.



\begin{thebibliography}{99}

\bibitem[D\'{\i}az-Garc\'{\i}a(2012)]{dg:12}
    D\'{\i}az-Garc\'{\i}a, J. A. (2012).
    Distributions on symmetric cones I: Riesz distribution.
    \texttt{http://arxiv.org/abs/1211.1746}. Also submited.

\bibitem[Faraut and Kor\'anyi(1994)]{fk:94}
    Faraut, J., and Kor\'anyi, A. (1994).
    \textit{Analysis on symmetric cones}.
    Oxford Mathematical Monographs,
    Clarendon Press, Oxford.

\bibitem[Gross and Richards(1987)]{gr:87}
    Gross, K. I., and  Richards, D. ST. P. (1987).
    Special functions of matrix argument I: Algebraic induction zonal polynomials
    and hypergeometric  functions.
    \textit{Trans. Amer. Math. Soc.}
    301(2), 475--501.

\bibitem[Hassairi and Lajmi(2001)]{hl:01}
   Hassairi, A., and Lajmi, S. (2001).
   Riesz exponential families on symmetric cones.
   \textit{J. Theoret. Probab.}, 14, 927--948.

\bibitem[Hassairi \textit{et al.}(2005)]{hlz:05}
   Hassairi, A., Lajmi, S., and Zine, R. (2005).
   Beta-Riesz distributions on symmetric cones,
   \textit{J. Statist. Plann. Inf.}, 133, 387--404.

\bibitem[Hassairi \textit{et al.}(2008)]{hlz:08}
   Hassairi, A., Lajmi, A., Zine, R. 2008.
   A chacterization of the Riesz probability distribution.
   J. Theoret. Probab. 21, 773-–790.

\bibitem[Ko{\l}odziejek(2014)]{k:14}
   Ko{\l}odziejek, B. 2014.
   The Lukacs-Olkin-Rubin theorem on symmetric cones without invariance of the ``Quotient".
   J. Theoret. Probab. DOI 10.1007/s10959-014-0587-3.

\bibitem[Khatri(1966)]{k:66}
    Khatri, C. G. (1966).
    On certain distribution problems based on positive definite quadratic
    functions in normal vector.
    \textit{Ann. Math. Statist.}
    37, 468--479.

\bibitem[Magnus and Neudecker(1988)]{mn:88}
   Magnus, J. R., and Neudecker, H. (1988).
   \textit{Matrix Differential Calculus with  Applications in Statistics and Econometrics}.
   John Wiley \& Sons, New York.

\bibitem[Muirhead(1982)]{m:82}
    Muirhead, R. J. (1982).
    \textit{Aspects of Multivariate Statistical Theory}.
    John Wiley \& Sons, New York.

\end{thebibliography}
\end{document}